\documentclass[11pt,letterpaper]{amsart}
\usepackage[utf8]{inputenc}
\usepackage{amsmath,amssymb,graphicx,setspace,verbatim,url}
\usepackage{color}
\usepackage{hyperref}
\usepackage{multicol}
\usepackage{enumitem}
\usepackage{float}
\usepackage{soul}
\usepackage{tikz}
\usepackage[normalem]{ulem}

\usepackage{xcolor}
\usetikzlibrary{matrix}
\usepackage{xargs}

\usepackage{tikz-cd}
\usetikzlibrary{arrows.meta}
\newenvironment{red}{\relax\color{red}}{\hspace*{.5ex}\relax}
\newcommand{\bcp}{\begin{red} }
\newcommand{\ecp}{\end{red}}

\usepackage{comment}

\tikzcdset{
  arrow style={tikz,diagrams={>=Triangle}}
}

\input xy
\xyoption{all}

\newtheorem{thmalpha}{Theorem}

\newtheorem{prop}{Proposition}[section]
\newtheorem{coro}[prop]{Corollary}
\newtheorem{thm}[prop]{Theorem}
\newtheorem{lemma}[prop]{Lemma}

\newtheorem{definition}[prop]{Definition}

\newcommand{\GG}{\mathcal{G}}

\newcommand{\DD}{\mathcal{D}}

\newcommand{\CC}{\mathcal{C}}

\newcommand{\Aut}{\mathrm{Aut}}

\newcommand{\D}{\mathcal{D}} 
\newcommand{\Cyc}{\mathrm{C}}

\newcommand{\C}{\mathcal{C}} 
\newcommand{\ZZ}{\mathbb{Z}}

\newcommand{\Br}{\operatorname{Br}}

\usepackage[loadshadowlibrary,shadow, colorinlistoftodos,textsize=tiny,obeyFinal]{todonotes}

\setuptodonotes{fancyline, backgroundcolor=gray!10,bordercolor=gray}

\newcommandx{\inline}[2][1=]{\todo[inline, #1]{#2}}
\newcounter{hw}
\newcommandx{\homework}[2][1=]{\todo[inline, caption={Homework \thehw} #1]{\stepcounter{hw} #2}}

\newcommandx{\claudioLong}[2][1=Long todo, usedefault]{\todo[inline, color=red!25,caption={\textbf{Claudio:} #1}]{\textbf{Claudio: #1}. #2}}
\newcommandx{\philippeLong}[2][1=Long todo, usedefault]{\todo[inline,color=purple!25, caption={\textbf{Philippe:} #1}]{\textbf{Philippe: #1}. #2} }

\begin{document}
\title{Two-orbit right-angled Polytopes and Hypertopes}

\author[Claudio Alexandre Piedade]{Claudio Alexandre Piedade}
\address[Claudio Alexandre Piedade]{Université Libre de Bruxelles, Brussels, Belgium
}
\email{claudio.piedade@ulb.be}

\author[Philippe Tranchida]{Philippe Tranchida}
\address[Philippe Tranchida]{Université Libre de Bruxelles, Brussels, Belgium
}
\email{tranchida.philippe@gmail.com}
{}

\setcounter{tocdepth}{1}

\thanks{\textit{Declaration of interests:} none}
\date{\today}

\begin{abstract}
    Abstract polytopes and hypertopes are thin and residually-connected incidence geometries or, equivalently, chamber systems. Their degree of regularity is understood by looking at the action of their automorphism group on the set of their chambers. When this action has only one orbit, they are called regular. In this article, we are interested in polytopes and hypertopes with exactly two orbits on their chambers. Let $I = \{1,2,\ldots,n\}$. For any $J\subsetneq I$ with $|I| \geq 3$, we show that there exists a class $2_J$ two-orbit hypertope whose diagram is any graph on $|I|$ vertices, with edges all labeled by $\infty$. In order to do so, we first establish a general framework for constructing hypertopes as quotient of chamber systems, and then exploit the machinery of acylindrically hyperbolic groups for showing that this framework can be used to create two-orbit polytopes and hypertopes of class $2_J$. Note that all polytopes and hypertopes constructed in this paper are infinite.
 
\end{abstract}

\maketitle
\section{Introduction}

Abstract polytopes form an area of mathematics where both algebra and geometry meet. The original definition of an abstract polytope can be found in~\cite{ARP} in terms of ordered posets whose elements are called faces. 
An equivalent way of defining an abstract polytope is by considering it as an incidence geometry which is thin, residually-connected and whose Buekenhout diagram is linear~\cite{buekenhout2013diagram}.
Recently, abstract polytopes have also been studied using the concept of chamber systems, mostly under the name of maniplexes, see for example \cite{ManiplexesWilson,PolytopalityManiplexes, SymmetryTypeGraph}. We will adopt the chamber system point of view in this paper.

A chamber system $\C$ consists of a set $C$ of elements, called chambers, together with a set of equivalence relations $(\sim_i)_{i\in I}$, where $I$ is a finite index set~\cite{buekenhout2013diagram}. Two chambers $c,d\in \C$ are said to be $i$-adjacent whenever $c\sim_i d$. Chamber systems can be seen as an edge-coloured graph, where the vertices correspond to chambers and two vertices $c$ and $d$ are joined by an edge of colour $i$ whenever $c \sim_i d$.
It is usual to study chamber systems which are highly symmetric. The most frequently studied class of highly symmetric objects are those which are regular or chiral. A chamber systems is said to be \emph{regular} when its automorphism group $\Aut(\C)$ is transitive on the set of chambers, and it is \emph{chiral} when there are exactly two orbits of the action of $G$ on $C$ such that any $i$-adjacent chambers are in different orbits, for every $i\in I$. 

In this paper, we focus on a less common class of symmetric structures, which generalizes the notion of chirality. For any $J\subset I$, we say that $\C$ is a \emph{two-orbit} chamber system of type $2_J$ if the action of $G$ on the chambers has two orbits and $i$-adjacent chambers are in distinct orbits if and only if $i\in I\setminus J$. Thinking of $\C$ as a graph, a two-orbit chamber system of type $2_J$ admits a bicoloring of its chambers such that $i$-adjacent chambers are of the same colour if and only if $i \in J$. A chiral polytope is then a two-orbit polytope of class $2_\emptyset$. Two-orbit polytopes, or maniplexes, have been studied before, see for instance \cite{Pellicer_2025,ExistanceTwoOrbitsManiplexes,Polytopality2orbitsmaniplexes}

Chamber systems can also be built from groups, in which case they are called \emph{coset chamber systems}. This notion is due to Jacques Tits and more details can be found in~\cite{buekenhout2013diagram}. In this paper we will use right-angled Coxeter groups to build chamber systems.
Let $(W,S)$ be a right-angled Coxeter system, meaning that $W$ is a right-angled Coxeter group generated by a finite set of involutions $S = \{s_i\mid i\in I\}$. The Coxeter chamber system of $(W,S)$, denoted $\C(W,S)$, is the chamber system whose chambers correspond to the elements of $W$, and two chambers $g,h\in W$ are $i$-adjacent if and only if $g\langle s_i\rangle = h\langle s_i\rangle$. For any $J \subset I$, the chambers of a right-angled Coxeter chamber systems can always be bi-coloured in such a way that two $i$-adjacent chambers have the same colour if and only if $i \in J$. We define $C_J$ to be the subgroup of $W$ that preserves the bi-colouring induced by $J$. Note that $C_J$ is always an index two subgroup of $W$ (see Section \ref{sec:RegTwoOrbHyper} for more details). 

In \cite{Pellicer_2025}, the author lists a long series of open question regarding chiral polytopes and related topics. One of these questions asks whether, for any finite index set $I$, there exists a two-orbit polytope of type $2_J$ for any proper subset $J \subset I$. 
This problem can be equally posed for two-orbit hypertopes, a generalisation of polytopes where the linear diagram condition is dropped.
In this paper, we give a positive answer to this more general question, showing that, for any finite type set $I$ and any given unlabelled diagram $D$ over $I$, two-orbit hypertopes of type $2_J$ exist. 

To achieve this goal, we work with quotients of coset chamber systems. Indeed, any hypertope can be realised as a quotient of the coset chamber system of its covering right-angled Coxeter group. More precisely, every hypertope can be seen as the chamber system $\C(W,S)/H$, for some subgroup $H\leq W = \Aut(\C(W,S))$, where $(W,S)$ is a well-chosen right-angled Coxeter system. The automorphism group of such a quotient is then isomorphic to $N_W(H)/H$. Our aim is to find a particular subgroup $H\leq W$ such that: (1) $\C(W,S)/H$ is a hypertope; and (2) $N_W(H) = C_J$.
To achieve condition (1), we prove a general result about quotients of coset chamber systems.
\begin{thmalpha}[Theorem~\ref{thm:quotient_RC}]\label{thm:A}
    Let $\CC=\CC\bigl(G,B,(G^{(i)})_{i\in I}\bigr)$
be a residually-connected coset chamber system over the finite type set $I$.  Let $N\trianglelefteq G$ such that $N\cap G^{(J)}G^{(K)}G^{(L)} = \{1\}$ for $J,K,L\subseteq I$. 
Then, for $H\leq N$, $\CC/H$ is
residually-connected.
\end{thmalpha}
Additionally, we prove that a quotient by a subgroup as in the above theorem also preserves residues, and hence the diagram of $\C$.

The second condition guarantees that $\C(W,S)/H$ is of class $2_J$.
Thanks to Theorem \ref{thm:A}, we will see that $\C(W,S)/H$ is automatically a hypertope whenever $H$ is chosen to be a subgroup of the Brunnian subgroup $\Br(W,S)$, where The Brunnian subgroup can be loosely be described as the set of elements of $W$ that become trivial after deleting any of the generators.
We chose the name Brunnian by analogy with Brunnian links and braids. These links and braid have the defining property that deleting any one of their strands renders them trivial. In order to find a subgroup $H \leq \Br$ satisfying $(2)$, we use the powerful machinery of acylindrically hyperbolic groups.

The core result of this paper is the following, whose proof is recorded here in order to give an overview. The rest of the paper is dedicated to established the core results needed in this proof. 

\begin{thmalpha}\label{thm:main}
    For every finite type set $I$, any proper subset $J \subsetneq I$ and any connected graph $\mathcal{G}$ with at least three vertices, there exists a class $2_J$ hypertope whose diagram $\mathcal{D}$ is $\mathcal{G}$ where every edge is labelled by $\infty$. In particular, there exist class $2_J$ two-orbit polytopes of any rank.
\end{thmalpha}

\begin{proof}
    Let $\mathcal{D}$ be the diagram obtained from the graph $\mathcal{G}$ by labeling all edges with $\infty$. Build the Coxeter group $W = W(\mathcal{D})$ and let $\C := \mathcal{C}(W,S)$ be its chamber system. Consider $\Br = \Br(W,S)$, the Brunnian subgroup of $W$. By Theorem \ref{thm:Brunnian-hypertope}, for any subgroup $H \subset \Br$, the quotient chamber system $\C/H$ is a hypertope. By Proposition \ref{prop:AutomorphismGroupQuotient}, the automorphism group of $\C/H$ is isomorphic to $N_W(H)/H$, where $N_W(H)$ denotes the normalizer in $W$ of $H$. Finally, by Theorem \ref{thm:prescribed-normalizer}, there exists such a subgroup $H \leq \Br$ such that $N_W(H) = C_J$. Therefore, $\mathcal{C}/H$ is the desired two-orbit hypertope of class $2_J$.
\end{proof}

Note that all of the polytopes and hypertopes constructed in this paper are infinite. They are what we call right-angled polytopes and hypertopes, because the only Coxeter relations that they satisfy are commutation of some generators. The harder question of finding finite two-orbit polytopes of type $2_J$ for any $J \subset I$ is thus still open.

\subsection{Structure of the Paper}

In section \ref{sec:background}, we gather all necessary background information on chamber systems, Coxeter groups, hypertopes and acylindrically hyperbolic groups and actions. Section \ref{sec:quotientChamberSystems} then establishes a general setting in which quotient of chamber systems behave nicely. In particular, we show that if Equation \eqref{eq:trivial_intersection_N} holds, the quotient chamber system inherits properties such as connectedness, residual-connectedness and thinness. Moreover, the residues of the starting coset chamber system and of the quotient are all isomorphic. In Section \ref{sec:brunnian}, we define the Brunnian subgroup of a right-angled Coxeter group and show that it satisfies all the conditions of the construction in Section \ref{sec:quotientChamberSystems}. Finally, Section \ref{sec:AppropriateSub} exploits the fact that a right-angled Coxeter group is acylindrically hyperbolic to show that we can find a subgroup $H$ of the Brunnian group $\Br \subset W$ such that $N_W(H) = C_J$, the required condition for constructing two-orbit hypertopes of class $2_J$.

\section{Background}\label{sec:background}

\subsection{Chamber systems}\label{sec:ChamberSystems}

Throughout this section, $I$ denotes a finite non-empty set of types.  We
follow the chamber-system terminology of
\cite[Chapter~3]{buekenhout2013diagram}.

A \emph{chamber system over $I$} is a pair $\CC=(C,(\sim_i)_{i\in I})$,
where $C$ is a non-empty set, whose elements are called \emph{chambers}, and
each $\sim_i$ is an equivalence relation on $C$.  Its \emph{rank} is $|I|$.
The equivalence classes of $\sim_i$ are called the \emph{$i$-panels}, and
chambers in the same $i$-panel are \emph{$i$-adjacent}.  The chamber system
is \emph{firm}, \emph{thick}, or \emph{thin} if every panel has, respectively,
at least two, at least three, or exactly two chambers.

From a chamber system $\C$, one can define its \emph{chamber graph} $\GG(\CC)$: it has vertex set $C$ and two distinct chambers
$c,d$ are joined by an edge labelled $i$ whenever $c\sim_i d$.  A chamber
system $\C$ is said to be \emph{connected} if this chamber graph is connected.

Let $J\subseteq I$.  A \emph{$J$-gallery} is a finite sequence of chambers
in which consecutive chambers are $j$-adjacent for some $j\in J$.  The set
of chambers joined to $c$ by a $J$-gallery is denoted by $cJ^*$ and is called
the \emph{$J$-cell} containing $c$.  With the adjacency relations indexed by
$J$, it is itself a chamber system of rank $|J|$.  We allow galleries of
length zero, so $c\varnothing^*=\{c\}$.  For $X\subseteq C$, put
\[
    XJ^*:=\bigcup_{x\in X}xJ^*.
\]
Thus $cJ^*K^*$ means $(cJ^*)K^*$.

A \emph{homomorphism} of chamber systems over $I$ is a map preserving each
$i$-adjacency.  An \emph{isomorphism} is a bijection for which both the map
and its inverse are homomorphisms.  An isomorphism from $\CC$ to itself is
an \emph{automorphism}. We write $\Aut(\CC)$ for the group of these
type-preserving automorphisms.  A group of automorphisms is
\emph{chamber-transitive} if it acts transitively on $C$.  In a connected
thin chamber system, an automorphism fixing one chamber fixes its unique
other chamber in each panel, and then every chamber by following
galleries.  Thus a chamber-transitive action on such a system is simply
transitive.

For $i\in I$, the $(I\setminus\{i\})$-cells are called \emph{$i$-objects}.
The following definition controls how these objects meet.

\begin{definition}[{\cite[Section~3.4]{buekenhout2013diagram}}]
A chamber system $\CC$ is \emph{residually-connected} if, for every
$J\subseteq I$ and every family $(Z_j)_{j\in J}$ of pairwise intersecting
$j$-objects, the intersection $\bigcap_{j\in J}Z_j$ is an
$(I\setminus J)$-cell.
\end{definition}
Note that, for $J=\varnothing$, the intersection in the definition above is understood
to be $C$, hence residual-connectedness implies connectedness.

A related, and more common concept, is the one of incidence systems \cite[Chapter 1]{buekenhout2013diagram}. Whenever a chamber system
$\CC$ is a residually-connected chamber system, there is a unique residually-connected
incidence geometry with the same chambers, residues and automorphism group
~\cite{buekenhout2013diagram}. Hence, residually connected chamber systems and residually connected incidence systems are basically the same thing, from two different points of view.

We will use the following standard characterization of residual-connectedness.

\begin{lemma}[{\cite[Lemma~3.4.9]{buekenhout2013diagram}}]
\label{lem:rc-characterisation}
Let $\CC$ be a non-empty connected chamber system over the finite set $I$.
Then $\CC$ is residually-connected if and only if, for every chamber $c$ and
all $J,K,L\subseteq I$,
\begin{equation}
\label{eq:rc-cell-identity}
    cJ^*L^*\cap cK^*L^*=c(J\cap K)^*L^*.
\end{equation}
\end{lemma}

A \emph{diagram} over $I$ is a labeled graph with vertex set $I$ that specifies, for each pair of distinct types
$i,j$, the allowed isomorphism types of the corresponding rank two cells.  We say that $\CC$
\emph{belongs to} the diagram if every $\{i,j\}$-cell is of an allowed
type, where the allowed types are recorded by the labels on the edge $\{i,j\}$. For a thin residually-connected chamber system, each rank two cell
must be an alternating cycle of length $2m$, with $m\geq2$, or a doubly infinite
path as its chamber graph.  The latter is denoted by $m=\infty$.  When $m$
is the same for every $\{i,j\}$-cell, we denote this common number as $m_{ij}$. An edge labeled $m_{i,j}$ on the diagram of a thin and residually-connected chamber system therefore means that all $\{i,j\}$-cells are alternating paths of length $2m_{i,j}$. We also omit edges labelled $2$.  A diagram is \emph{linear} if the types can be ordered so that
$m_{ij}=2$ whenever $|i-j|>1$.  These are the usual diagram conventions for a Buekenhout diagram for a chamber system,
expressed in terms of cells rather than residues of an incidence geometry;
see \cite[Sections~3.5 and~4.1]{buekenhout2013diagram}.

Let $\CC=(C,(\sim_i)_{i\in I})$ be a chamber system, and let
$H\leq\Aut(\CC)$.  The \emph{quotient chamber system} $\CC/H$ has chamber set
$C/H$, the set of $H$-orbits on $C$, and adjacency relations
\begin{equation*}
    Hc\sim_i Hd
    \quad\Longleftrightarrow\quad
    hc\sim_i d\text{ for some }h\in H.
\end{equation*}
Since the relations $\sim_i$ are $H$-invariant equivalence relations, the
relations defined above are still
equivalence relations (see \cite[Proposition~3.1.10]{buekenhout2013diagram}).

The automorphism group of the quotient chamber system is easily controlled whenever it is thin.

\begin{prop}\label{prop:AutomorphismGroupQuotient}
    Suppose that $\C$ is a thin chamber system with automorphism group $G$ and let $H$ be a subgroup of $G$ such that $\C/H$ is a thin chamber system. Then, the automorphism group of the quotient chamber system $\C/H$ is isomorphic to $N_G(H)/H$.
\end{prop}
\begin{proof}
    This is proven in \cite[Proposition 4.2.10]{buekenhout2013diagram} when $\C$ is the chamber system constructed from a Coxeter group. The proof follows the exact same arguments for this more general setting.
\end{proof}

\subsection{Coset chamber systems}\label{sec:CosetChamberSystem}

One important concept in both incidence and chamber systems is the fact that both structures can often be built starting from a group. Here, we recall the construction of a chamber system from a group.

Let $G$ be a group, let $B\leq G$, and let $(G^{(i)})_{i\in I}$ be a family
of subgroups of $G$ such that $B\leq G^{(i)}$ for every $i\in I$.  The
\emph{coset chamber system}
    $\CC\bigl(G,B,(G^{(i)})_{i\in I}\bigr)$
has chamber set $G/B=\{gB\mid g\in G\}$, with equivalence relations defined as
\[
    gB\sim_i hB
    \quad\Longleftrightarrow\quad
    gG^{(i)}=hG^{(i)}.
\]
The subgroup $B$ is called the \emph{Borel subgroup}, and the $G^{(i)}$ are the
\emph{standard minimal parabolic subgroups}.

Left multiplication gives a chamber-transitive action of $G$, with
$B$ being the stabilizer of the chamber $B$.  An $i$-panel of a chamber $gB$ consists of the set of chambers
$gxB$ with $x\in G^{(i)}$, and has size $[G^{(i)}:B]$.  In particular, the
chamber system is thin exactly when $[G^{(i)}:B]=2$ for every $i\in I$.

For $J\subseteq I$, define
    \[G^{(J)}:=\bigl\langle B,G^{(j)}\mid j\in J\bigr\rangle.\]
In particular, $G^{(\varnothing)}=B$.  The $J$-cell containing $wB$ is
\begin{equation}
\label{eq:J-cell-coset-system}
    (wB)J^*=\{wxB\mid x\in G^{(J)}\}.
\end{equation}
The definition of connected and residually-connected can be verified by checking properties of the subgroups $G^{(J)}$, for $J\subseteq I$.

\begin{lemma}[{\cite[Lemma~3.6.7]{buekenhout2013diagram}}]
\label{lem:connected-coset-system}
The coset chamber system $\CC(G,B,(G^{(i)})_{i\in I})$ is connected if and only if
\[
    G=G^{(I)}=\langle G^{(i)}\mid i\in I\rangle.
\]
\end{lemma}

\begin{prop}[{\cite[Theorem~3.6.9]{buekenhout2013diagram}}]
\label{prop:RC_CosetChamberSystem}
Let $I$ be finite, let $B\leq G^{(i)}\leq G$ for every $i\in I$, and assume
that $G=\langle G^{(i)}\mid i\in I\rangle$.  Then the following statements
are equivalent.
\begin{enumerate}
    \item The chamber system $\CC(G,B,(G^{(i)})_{i\in I})$ is residually-connected.
    \item For all $J,K,L\subseteq I$,
    $G^{(L)}\cap G^{(J)}G^{(K)}=G^{(L\cap J)}G^{(L\cap K)}$.
    \item For all $J,K,L\subseteq I$,
    $G^{(J)}G^{(L)}\cap G^{(K)}G^{(L)}=G^{(J\cap K)}G^{(L)}$.
\end{enumerate}
\end{prop}

\subsection{Coxeter groups}

A \emph{Coxeter system} is a pair $(W,S)$, with $S=\{r_i\mid i\in I\}$,
such that
\[
    W=\left\langle r_i\ (i\in I)\ \middle|\
       (r_ir_j)^{m_{ij}}=1\text{ whenever }m_{ij}<\infty\right\rangle,
\]
where $m_{ii}=1$ and, for $i\neq j$, $m_{ij}=m_{ji}$ belongs to
$\{2,3,\ldots\}\cup\{\infty\}$.  Thus each $r_i$ is an involution, and
$m_{ij}=\infty$ means that no relation is imposed on the order of $r_ir_j$.  For
$J\subseteq I$, write $S_J=\{r_j\mid j\in J\}$ and
$W_J=\langle S_J\rangle$, the \emph{standard parabolic subgroup of type
$J$}, with $W_\varnothing=1$; see \cite[Chapter~IV, Section~1]{Bourbaki2006}.

\begin{prop}[{\cite[Chapter~IV, Section~1, Theorem~2 and Exercise~1(b)]{Bourbaki2006}}]
\label{prop:CoxeterProperties}
Let $(W,S)$ be a Coxeter system.  Then
\begin{enumerate}
    \item For $J\subseteq I$, the pair $(W_J,S_J)$ is a Coxeter system
    with the restricted Coxeter matrix.
    \item For $J,K\subseteq I$, we have $W_J\cap W_K=W_{J\cap K}$.
    \item For $J,K,L\subseteq I$, we have
    $W_JW_L\cap W_KW_L=W_{J\cap K}W_L$.
\end{enumerate}
\end{prop}

The \emph{Coxeter diagram} $\DD(W,S)$ has vertex set $I$ and an edge
labelled $m_{ij}$ between $i$ and $j$ precisely when $m_{ij}>2$. In particular, a missing edge means that the
two generators commute.  The system $(W,S)$ is \emph{irreducible} if
$\DD(W,S)$ is connected.

The \emph{Coxeter chamber system} of $(W,S)$ is
\[
    \CC(W,S):=\CC\bigl(W,\langle 1\rangle,(\langle r_i\rangle)_{i\in I}\bigr).
\]
Its chambers are the elements of $W$, and the $i$-panel through $w$ is
$\{w,wr_i\}$.  Here $W^{(i)}=\langle r_i\rangle$ and $W^{(J)}=W_J$, so the
$J$-cell through a chamber $w$ corresponds to the coset $wW_J$.  This system is thin,
residually-connected, and $W$ acts simply transitively on it by left
multiplication.  The diagram of the chamber system is denoted $\DD(W,S)$. A $\{i,j\}$-cell is the chamber
system of the dihedral group $W_{\{i,j\}}$, with $2m_{ij}$ chambers when
$m_{ij}<\infty$ and a doubly infinite path when $m_{ij}=\infty$; see
\cite[Sections~4.1-4.2 and Corollary~4.5.16]{buekenhout2013diagram}.

A Coxeter system is \emph{right-angled} if $m_{ij}\in\{2,\infty\}$ for
all distinct pairs $i,j\in I$.  Writing $\Delta=\DD(W,S)$, a right-angled Coxeter group has the following presentation:
\[
    W=\left\langle r_i\ (i\in I)\ \middle|\
        r_i^2=1,\ [r_i,r_j]=1
        \text{ if }i\neq j\text{ and }\{i,j\}\notin E(\Delta)\right\rangle.
\]
Here and below, let $[x,y]:=x^{-1}y^{-1}xy$. An alternative way of obtaining a right-angled Coxeter groups is by using graph products.
For a finite simple graph $\Gamma$ and groups $(A_i)_{i\in V(\Gamma)}$, the
\emph{graph product} $\Gamma\{A_i\}$ is obtained from the free product
$\ast_i A_i$ by imposing $[a,b]=1$ whenever $a\in A_i$, $b\in A_j$, and
$\{i,j\}\in E(\Gamma)$. Thus a right-angled Coxeter group is the graph
product of copies of $\Cyc_2$ over the \emph{commutation graph}
$\Gamma=\Delta^c$.  Notice the opposite edge conventions: edges in
$\Gamma$ mean commutation, whereas edges in $\Delta$ mean non-commutation.
A graph product is called \emph{irreducible} when $\Gamma^c$ is connected,
which corresponds with irreducibility of the Coxeter system.

\subsection{Regular and two-orbit hypertopes}\label{sec:RegTwoOrbHyper}

Abstract polytopes and hypertopes are most commonly defined via incidence systems. In this article, we focus on the chamber system perspective, which is equivalent.

A \emph{hypertope} is a thin and residually-connected chamber system over a finite type set $I$. If the diagram of an hypertope is linear, we say it is a \emph{polytope}.

A hypertope is \emph{regular} if its automorphism group has a single orbit on its chambers. A \emph{two-orbit} hypertope has precisely two orbits of chambers under the action of its automorphism group. 

Let $\mathcal{H}$ be a two-orbit hypertope and let $c \in H$  be a chamber. For every $i \in I$, there is a unique $i$-adjacent chamber $c^i$ to $c$. If for any chamber $c$, the chamber $c^i$ is in the same orbit as $c$, then it will be true for any other chamber, meaning that any chamber $\Bar{c}$ and $\Bar{c}^i$ will be in the same orbit. This is proven for polytopes in \cite[Proposition 2.219]{Pellicer_2025}, however the exact same proof works for hypertopes. Therefore, two-orbit hypertopes can be classified by the indices $i \in I$ for which $c$ and $c^i$ are in the same orbit. For any $J \subset I$, we say that $\mathcal{H}$ is of \emph{type} $2_J$ whenever, for any chamber $c$, it holds that $c$ and $c^j$ are in the same orbit if and only if $j \in J$.

Let $(W,S)$ be a right-angled Coxeter system. The Coxeter chamber system $\C(W,S)$ is always a regular hypertope over $I = |S|$ and all hypertopes can be constructed as quotients of a right-angled Coxeter chamber system. For any choice $J \subset I$, there exists a well-defined group homomorphism $\chi_J \colon W \to C_2$ defined by
$$
\chi_J(r_i) = \begin{cases}0,&i\in J,\\1,&i\notin J.\end{cases}
$$
One should think of $\chi_J$ in terms of colouring of the chamber system. For $i \in J$, chambers that $i$-adjacent have the same colour, whereas for $i \notin J$, they have different colours. The \emph{$J$-colours preserving subgroup} of $W$ is $C_J := \ker(\chi_J)$.

 Whenever a hypertope $\mathcal{H}$ is described as the quotient $\C(W,S)/H$ for some subgroup $H \leq W$, we can then algebraically verify whether $\mathcal{H}$ is regular or of class $2_J$. Indeed, $\mathcal{H}$ is regular if and only if $H$ is a normal subgroup and $\mathcal{H}$ is of class $2_J$ exactly when $N_W(H) = C_J$. Indeed, let $c$ and $d$ be two chambers of $\mathcal{H}=\C(W,S)/H$ that are $j$-adjacent. Then, by definition of quotient chamber systems, there exist two chamber $\hat{c}$ and $\hat{d}$ of $\C(W,S)$ that are also $j$-adjacent and that are lifts of $c$ and $d$, respectively. As $\C(W,S)$ is regular, there exists a $\varphi \in \Aut(\C(W,S))$ such that $\varphi(\hat{c}) = \hat{d}$. Clearly, $\varphi \in C_J$ if and only if $j \in J$. Therefore, as $\Aut(\mathcal{H}) = N_W(H)/H = C_J/H$ by assumption, the automorphism $\varphi$ induces an automorphism of $\mathcal{H}$ if and only if $j \in J$, so that $c$ and $d$ are in the same orbit if and only if $j \in J$.

\subsection{Acylindrically hyperbolic groups} \label{sec:AcylindricallyHyperbolic}

In section~\ref{sec:AppropriateSub}, we will use results of acylindrically hyperbolic actions and groups, following 
\cite[Sections~1 and~2]{OsinAcylindrical} and
\cite[Section~6.1]{DGO}. Here, we lay down the necessary terminology and results to follow that section. Assume that all actions in this subsection are by isometries.

A geodesic metric space $(X,d)$ is \emph{hyperbolic} if there is a constant
$\delta\geq0$ such that every side of every geodesic triangle is contained
in the $\delta$-neighbourhood of the other two sides.  We write $\partial X$
for its Gromov boundary.

\begin{definition}[Acylindrical action]
\label{def:acylindrical}
An action of a group $G$ on a metric space $(X,d)$ is \emph{acylindrical} if,
for every $\varepsilon>0$, there exist $R\geq0$ and an integer $N\geq1$ such
that, whenever $x,y\in X$ satisfy $d(x,y)\geq R$, we have
\[
    \bigl|\{g\in G\mid d(x,gx)\leq\varepsilon,
                       \ d(y,gy)\leq\varepsilon\}\bigr|\leq N.
\]

\end{definition}
In other words, in an acylindrical action there is a uniform bound on the number of group elements that move
two sufficiently distant points by at most $\varepsilon$.
\begin{definition}[Loxodromic element]
\label{def:loxodromic}
Let $G$ act on a hyperbolic space $X$.  An element $h\in G$ is
\emph{loxodromic} if, for some $x\in X$, the orbit map
\[
    \ZZ\longrightarrow X,\qquad m\longmapsto h^m x,
\]
is a quasi-isometric embedding.  Explicitly, there are constants
$\lambda\geq1$ and $c\geq0$ such that
\[
    \lambda^{-1}|m-n|-c
    \leq d(h^m x,h^n x)
    \leq\lambda|m-n|+c
    \qquad(m,n\in\ZZ).
\]

\end{definition}

A loxodromic element always has infinite order and fixes two distinct points
$h^+,h^-\in\partial X$ on the boundary of $X$. These fixed points are obtained as the limits of $h^m x$ as
$m\to+\infty$ and $m\to-\infty$, respectively. Loxodromic elements should be thought as translation along an infinite axis.  Two loxodromic elements
are called \emph{independent} if their pairs of fixed points are disjoint,
see \cite[Section~1]{OsinAcylindrical}.

We now define an non-elementary action as an action having sufficiently interesting dynamics.
\begin{definition}[Non-elementary action]
\label{def:elementary}
Let $G$ act on a hyperbolic space $X$.  Its \emph{limit set}
$\Lambda_X(G)\subseteq\partial X$ is the set of accumulation points in
$\partial X$ of orbits $G \cdot x$.
The action is \emph{elementary} if $|\Lambda_X(G)|\leq2$, and
\emph{non-elementary} if $|\Lambda_X(G)|>2$.
\end{definition}

For an acylindrical action, non-elementarity is equivalent to the existence
of two independent loxodromic elements.  Equivalently, the action has
unbounded orbits and $G$ is not virtually cyclic, where a group is
\emph{virtually cyclic} if it has a cyclic subgroup of finite index;
see \cite[Theorem~1.1]{OsinAcylindrical}. 

We are now ready to define acylindrically hyperbolic groups following \cite[Definition~1.3]{OsinAcylindrical}.
\begin{definition}[Acylindrically hyperbolic]
\label{def:AH}
A group $G$ is \emph{acylindrically hyperbolic} if it admits a
non-elementary acylindrical action on a hyperbolic space.
\end{definition}

Note that the definition of acylindrically hyperbolic groups comes with a specified action. Other actions might not be suitable.

\begin{definition}[WPD element~{\cite[Definition~6.1]{DGO}}]
\label{def:WPD}
Let $G$ act on a hyperbolic space $X$.  An element $h\in G$ satisfies the
\emph{weak proper discontinuity condition}, or is a \emph{WPD element},
if, for every $\varepsilon>0$ and every $x\in X$, there exists an integer
$M\geq1$ such that
\[
    \{g\in G\mid d(x,gx)\leq\varepsilon,
                  \ d(h^M x,gh^M x)\leq\varepsilon\}
\]
is finite.
\end{definition}

Here $M$ may depend on both $x$ and $\varepsilon$.  Every loxodromic element
of an acylindrical action is WPD: choose $M$ with $d(x,h^M x)\geq R$ and
apply Definition~\ref{def:acylindrical}. 

\begin{prop}[{\cite[Lemma~6.5 and Corollary~6.6]{DGO}}]
\label{prop:EG}
Let $h$ be a loxodromic WPD element for an action of $G$ on a hyperbolic
space $X$.  There is a unique maximal virtually cyclic subgroup of $G$
containing $h$, denoted by $E_G(h)$ and called its \emph{maximal elementary
subgroup}.  It is given by
\[
    E_G(h)=\{g\in G\mid
       g^{-1}h^m g=h^{\pm m}\text{ for some integer }m\geq1\}.
\]
Equivalently, $g\in E_G(h)$ if and only if
\[
    g^{-1}h^m g=h^n
    \quad\text{for some }m,n\in\ZZ\setminus\{0\}.
\]
Geometrically, $E_G(h)$ is the setwise stabilizer of
$\{h^+,h^-\}$ in $G$.
\end{prop}

In particular, $\langle h\rangle$ has finite index in $E_G(h)$.
If $K\leq G$ contains $h$, restriction of the action preserves
loxodromicity and WPD, and the displayed characterization gives
\[
    E_K(h)=K\cap E_G(h).
\]

\begin{definition}[Non-commensurable~{\cite[Section~2]{hull2017small}}]
\label{def:commensurable}
Two infinite-order elements $a,b\in G$ are \emph{commensurable in $G$} if
there exist $g\in G$ and $m,n\in\ZZ\setminus\{0\}$ such that
\[
    g^{-1}a^m g=b^n.
\]
Otherwise they are \emph{non-commensurable in $G$}.
\end{definition}

We finish this section by citing a few results that will be needed later, rephrased in a convenient way.

\begin{thm}\cite[Theorem 6.8 and Theorem 7.19, adapted]{DGO}\label{thm:DGO_thms}
    Let $G$ be a group acting on a hyperbolic space $(X,d)$ and let $\{h_1,\ldots,h_k\}$ be a collection of pairwise non-commensurable loxodromic WPD elements of $G$. Then:
    \begin{enumerate}
        \item the collection $\{E(h_1),\ldots,E(h_k)\}$ hyperbolically embeds in $G$;
        \item there exists a family of finite subsets $F_i\subseteq E(h_i)\setminus \{1\}$ such that for every collection $N=\{N_i \triangleleft E(h_i)\mid 1\leq i\leq k\}$ satisfying $N_i \cap F_i=\emptyset$, the natural map from $E(h_i)/N_i$ to $G$ is injective for every $1\leq i \leq k$.
    \end{enumerate}
\end{thm}

We say a subgroup $H\leq G$ is \emph{$s$-normal} if $|H\cap H^g|=\infty$ for all $g\in G$.
\begin{lemma}\cite[Lemma 7.1]{OsinAcylindrical}\label{lem:snormal_non_elem}
    Let $G$ be a group acting acylindrically and non-elementarily on a hyperbolic space $S$. Then every $s$-normal subgroup of $G$ acts non-elementarily.
\end{lemma}

\section{Quotients of coset chamber systems} \label{sec:quotientChamberSystems}

Let $I$ be finite and $\CC=\CC\bigl(G,B,(G^{(i)})_{i\in I}\bigr)$ be a coset chamber system. Moreover, let $H\leq G\leq \Aut(\CC)$, acting on $\CC$ by left multiplication.  The chambers of
$\CC/H$ are the double cosets $HgB$, and $i$-adjacency is given by\begin{equation*}
HcB \sim_i HdB \Leftrightarrow \exists h\in H: hcB \sim_i dB,\ \textnormal{for }cB,dB\in C.\end{equation*}
Notice that the $i$-adjacency can be described equivalently as 
\begin{equation}
\label{eq:quotient-double-coset-adjacency-simp}HcB \sim_i HdB \Longleftrightarrow HcG^{(i)}=HdG^{(i)}.
\end{equation}

When $H\trianglelefteq G$, this quotient turn out to be isomorphic to the coset chamber system
\[\CC\bigl(G/H,BH/H,(G^{(i)}H/H)_{i\in I}\bigr).\]
For non-normal subgroups $H$, the natural description is instead in terms of double cosets. In either case, the following lemma describes the $J$-cells going through a chamber $HwB$ and the chambers of $(HwB)J^*B^*$.

\begin{lemma}
\label{lem:quotient-cell-description}
Let $w\in G$ and $J,K\subseteq I$.  Then
\begin{enumerate}
    \item $(HwB)J^* =\{HwxB\mid x\in G^{(J)}\}$;
    \item $(HwB)J^*K^* =\{HwxyB\mid x\in G^{(J)},\ y\in G^{(K)}\}$.
\end{enumerate}
\end{lemma}

\begin{proof}
By \eqref{eq:quotient-double-coset-adjacency-simp}, an $i$-adjacency from $HwB$
leads to a chamber of the form $HwxB$ with $x\in G^{(i)}$, and every such
chamber is $i$-adjacent to $HwB$.  Iterating this gives that any chamber reached by a $J$-gallery starting in $(HwB)$ must have the form given in $(a)$. Applying the same argument for the chambers of $(HwB)J^*K^*$ proves $(b)$.
\end{proof}

Let $N\trianglelefteq G$.  We say that $N$ satisfies the
\emph{three-factor avoidance condition} if
\begin{equation}
\label{eq:trivial_intersection_N}
    N\cap G^{(J)}G^{(K)}G^{(L)}=\{1\}
    \qquad\text{for all }J,K,L\subsetneq I.
\end{equation}
Empty subsets and repetitions are
allowed.  In particular, by taking two of the three subsets to be empty,
\eqref{eq:trivial_intersection_N} implies
\begin{equation}
\label{eq:one-factor-avoidance}
    N\cap G^{(J)}=\{1\}
    \qquad\text{for every }J\subsetneq I.
\end{equation}

Consider the following quotient map 
\[
    q_H\colon\CC\longrightarrow\CC/H,
    \qquad q_H(gB)=HgB.
\]
We show that $q_H$ induces isomorphisms of $J$-cells.
\begin{prop}
\label{prop:correspondenceResidues}
Assume that $N\trianglelefteq G$ satisfies
\eqref{eq:trivial_intersection_N}, and let $H\leq N$.  Fix $w\in G$.  For every
proper subset $J\subsetneq I$, the restriction
\begin{equation*}
    q_H\big|_{(wB)J^*}\colon (wB)J^*\longrightarrow(HwB)J^*
\end{equation*}
is an isomorphism of chamber systems over $J$.
\end{prop}

\begin{proof}
Surjectivity follows from \eqref{eq:J-cell-coset-system} and Lemma~\ref{lem:quotient-cell-description}$(a)$.  Suppose that
\[
    HwxB=HwyB
    \qquad(x,y\in G^{(J)}).
\]
Then there exist $h,h'\in H$ and $b,b'\in B$ such that
\[
    hwxb=h'wyb',
\]
which is equivalent to
\[
 h'^{-1}hwx=wyb'b^{-1}.
\]
Consequently,
\[
w^{-1}h'^{-1}hw=yb'b^{-1} x^{-1}.
\]
Since $H\leq N$ and $N\trianglelefteq G$, the left-hand side belongs to $N$. As $yb'b^{-1}x^{-1}\in G^{(J)}$, by
Equation~\eqref{eq:one-factor-avoidance} we have $h'^{-1}h=1$, or equivalently $h'=h$. Hence,
$xB=yB$ and thus $q_H\big|_{(wB)J^*}$ is injective.

It remains to show that adjacencies are the same. 
Suppose that $wxB\sim_i wyB$, for $i\in J$ and $wxB,wyB\in (wB)J^*$. Then by the $i$-adjacency of the quotient, $HwxB\sim_i HwyB$.
Conversely, let $i\in J$ and
suppose that $HwxB\sim_i HwyB$.  By
Equation~\eqref{eq:quotient-double-coset-adjacency-simp}, there exist $h,h'\in H$ and
$g,g'\in G^{(i)}$ such that
\[
    hwxg=h'wyg'.
\]
It follows that
\[
w^{-1}h'^{-1}hw=yg'g^{-1}x^{-1}.
\]
As before, this forces $h=h'$, meaning $xg = yg'$.  Hence
$xG^{(i)}=yG^{(i)}$, so $wxB\sim_i wyB$ in $\CC$.  Therefore the restriction
of $q_H$ is an isomorphism of chamber systems over $J$.
\end{proof}

\begin{coro}
\label{cor:quot_transfer_thin_thick_linearity}
Let $|I|\geq2$, $N\trianglelefteq G$ a normal subgroup that satisfies
the three-factor avoidance condition, and consider $H\leq N$.
\begin{enumerate}
    \item If $\CC$ is connected, then $\CC/H$ is connected;
    \item If $\CC$ is firm, thin, or
thick, then $\CC/H$ is firm, thin, or thick, respectively;
    \item If $|I|\geq 3$ and $\CC$ belongs to a diagram $\DD$, then $\CC/H$ belongs to the same diagram. In particular, a linear diagram remains linear after taking the
quotient.
\end{enumerate}
\end{coro}

\begin{proof}
$(a)$ The image under $q_H$ of a gallery is a gallery, so connectedness passes to the
quotient.  

$(b)$ For $i\in I$, the singleton $\{i\}$ is a proper subset of $I$.
Proposition~\ref{prop:correspondenceResidues} therefore identifies each
$i$-panel of $\CC/H$ with the corresponding $i$-panel of $\CC$, proving
that properties such as firmness, thinness, and thickness transfer to the quotient chamber system.

$(c)$ If $|I|\geq3$, each subset $\{i,j\}\subset I$ is proper. Proposition~\ref{prop:correspondenceResidues}
identifies every rank two cell in the quotient with a rank two cell
of $\CC$.  Hence, as the Buekenhout diagram is defined on the rank two cells, the diagram itself is
unchanged. In particular, if the diagram is linear, it will remain linear.
\end{proof}

We now show that the quotients we consider preserve residual-connectedness.

\begin{thm}
\label{thm:quotient_RC}
Let $ \CC=\CC\bigl(G,B,(G^{(i)})_{i\in I}\bigr)$
be a residually-connected coset chamber system over the finite type set $I$.  Let $N\trianglelefteq G$ satisfy
the three-factor avoidance condition, and consider $H\leq N$.  Then $\CC/H$ is
residually-connected.
\end{thm}

\begin{proof}
The quotient chamber system $\C/H$ is connected because $\C$ is connected.  Fixing now a chamber
$c=HwB$ of $\CC/H$ and subsets $J,K,L\subseteq I$,  we will prove the identity in
Lemma~\ref{lem:rc-characterisation}.
The direct inclusion is immediate and hence, we will prove the reverse inclusion
\[
    c(J\cap K)^*L^*\supseteq cJ^*L^*\cap cK^*L^*.
\]
 If $J=I$ or $K=I$, this inclusion is also immediate.  Additionally, if
$L=I$, all three sets are the whole quotient chamber system, hence we have trivially the inclusion. 

Let us then assume that $J,K,L\subsetneq I$ and consider $z\in cJ^*L^*\cap cK^*L^*$.  By
Lemma~\ref{lem:quotient-cell-description}, there exist
$x\in G^{(J)}$, $y\in G^{(K)}$, and $u,v\in G^{(L)}$
such that
\[
    z=HwxuB=HwyvB.
\]
Thus, by the equality of double-cosets, there exists $h,h'\in H$ and $b,b'\in B$ such that
$hwxub=h'wyvb'.$
Rearranging this equality gives
\begin{equation*}
    w^{-1}h'^{-1}hw=yvb'b^{-1}u^{-1}x^{-1}.
\end{equation*}
Since $y\in G^{(K)}$, $vb'b^{-1}u^{-1}\in G^{(L)}$ and $x\in G^{(J)}$, the right-hand side of the 
equality above is an element of
$G^{(K)}G^{(L)}G^{(J)}$. Similar to arguments before, the left-hand side belongs to $N$, and hence, by the three-factor avoidance condition given by Equation~\eqref{eq:trivial_intersection_N},
$h=h'$. Thus, $w^{-1}h'^{-1}hw=1$ and $yvb'= xub$.

As $\C$ is a residually-connected coset chamber system, by Proposition~\ref{prop:RC_CosetChamberSystem}, we have $G^{(J)}G^{(L)} \cap G^{(K)}G^{(L)} =G^{(J\cap K)}G^{(L)}$. Since $yvb'=xub\in G^{(J)}G^{(L)} \cap G^{(K)}G^{(L)}$, we have $yvb'=xub\in G^{(J\cap K)}G^{(L)}$.
Considering Lemma~\ref{lem:quotient-cell-description}, this means $z\in c(J\cap K)^*L^*$, finalizing the proof.
Hence $\CC/H$ is
residually-connected.
\end{proof}

This leads us to the following corollary, which can be seen as a generalization of \cite[Theorem 3.2]{fernandes2020exploration}.
\begin{coro}
Let $\C$ be a regular hypertope with diagram $\D$ and automorphism group $G := \langle \rho_0,\ldots,\rho_{n-1}\rangle$, with $n\geq 3$. Let $N\triangleleft G$ be a normal subgroup of $G$ such that it satisfies the three-factor avoidance condition. Then $\C/N$ is a regular hypertope with diagram $\D$.
\end{coro}

\begin{proof}
    Theorem \ref{thm:quotient_RC} and Corollary \ref{cor:quot_transfer_thin_thick_linearity} show that $\C/N$ is a residually-connected thin chamber system. Since $N$ is normal, Proposition \ref{prop:AutomorphismGroupQuotient} shows that $\Aut(\C/N)$ is isomorphic to $G/N$ acting transitively on the coset of $N$ in $G$. Hence, $\C/N$ is chamber transitive, and thus also a regular hypertope. Finally, the fact that the diagrams are equal is shown in Corollary \ref{cor:quot_transfer_thin_thick_linearity}.
\end{proof}

\section{The Brunnian subgroup of a right-angled Coxeter group}\label{sec:brunnian}

Let $I$ be finite with $|I|\geq2$, and let $(W,S)$ be a right-angled Coxeter
system, where $S=\{r_i\mid i\in I\}$.  Thus
\[
    W=\left\langle r_i\ (i\in I)\ \middle|\
        r_i^2=1,\ (r_ir_j)^2=1\text{ whenever }m_{ij}=2
      \right\rangle,
\]
with $m_{ij}\in\{2,\infty\}$ for $i\neq j$.

For every $J\subseteq I$, define the map $\pi_J\colon W\to W$ on the Coxeter
generators by
\begin{equation}
\label{eq:deletion-retraction}
    \pi_J(r_i)=
    \begin{cases}
        r_i,&i\in J,\\
        1,&i\notin J.
    \end{cases}
\end{equation}
This map is easily checked to extend to a homomorphism.
Moreover,
\begin{equation}
\label{eq:retraction-identities}
    \pi_J\circ\pi_K=\pi_{J\cap K}
    \qquad\text{and}\qquad
    \pi_J(W_K)=W_{J\cap K}.
\end{equation}

\begin{definition}
The \emph{Brunnian subgroup} of the right-angled Coxeter system $(W,S)$ is
\begin{equation*}
    \Br(W,S)
      :=\bigcap_{i\in I}\ker\bigl(\pi_{I\setminus\{i\}}\bigr).
\end{equation*}
\end{definition}

This subgroup is called the Brunnian subgroup of $(W,S)$ in analogy with Brunnian links and braids. Indeed, these links and braids have the defining property that deleting any one of their strands renders the  whole structure trivial. A famous example of this is the Borromean rings. 
Similarly, the elements if $\Br(W,S)$ are exactly those that become trivial whenever at least one $r_i$ is deleted.
Additionally, as $\Br(W,S)$ is an intersection of normal subgroups, $\Br(W,S)$ is also normal in $W$.

\begin{lemma}
\label{lem:pi(Br)trivial}
If $J\subsetneq I$, then
\[
    \pi_J\bigl(\Br(W,S)\bigr)=1.
\]
Equivalently, $\Br(W,S)\leq\ker(\pi_J)$.
\end{lemma}

\begin{proof}
Choose $i\in I\setminus J$. As $W_J\leq W_{I\setminus\{i\}}$, by
Equation~\eqref{eq:retraction-identities} we have
\[
    \pi_J=\pi_J\circ\pi_{I\setminus\{i\}}.
\]
The map $\pi_{I\setminus\{i\}}$ is trivial on $\Br(W,S)$ by definition, thus the
same is true of $\pi_J$.
\end{proof}

\begin{prop}
\label{prop:Brun_intersect}
Let $J,K,L\subsetneq I$.  Then
\begin{enumerate}
    \item $\Br(W,S)\cap W_JW_K=\{1\}$;
    \item $\Br(W,S)\cap W_JW_KW_L=\{1\}$.
\end{enumerate}
\end{prop}

\begin{proof}
\emph{(a)} Let $x\in\Br(W,S)\cap W_JW_K$, and write
$x=jk$ with $j\in W_J$ and $k\in W_K$.  Applying $\pi_J$ and using
Lemma~\ref{lem:pi(Br)trivial} gives
    $1=\pi_J(x)=j\pi_J(k)$,
so
\[
    j=\pi_J(k)^{-1}\in\pi_J(W_K)=W_{J\cap K}\leq W_K.
\]
Thus both $j$ and $k$ belong to $W_K$.  Since $K\subsetneq I$,
Lemma~\ref{lem:pi(Br)trivial} also gives
\[
    1=\pi_K(x)=jk=x.
\]

\emph{(b)} Let $x\in\Br(W,S)\cap W_JW_KW_L$, and write
$x=jkl$ with $j\in W_J$, $k\in W_K$, and $l\in W_L$.  Applying the three
retractions and using Lemma~\ref{lem:pi(Br)trivial}, we obtain
\begin{align*}
    j&=\pi_J(l)^{-1}\pi_J(k)^{-1}
       \in W_{J\cap L}W_{J\cap K},\\
    k&=\pi_K(j)^{-1}\pi_K(l)^{-1}
       \in W_{J\cap K}W_{K\cap L},\\
    l&=\pi_L(k)^{-1}\pi_L(j)^{-1}
       \in W_{K\cap L}W_{J\cap L}.
\end{align*}
Choose factorizations
\[
    j=ab,\qquad k=b'd,\qquad l=d'a',
\]
where
\[
    a,a'\in W_{J\cap L},\qquad
    b,b'\in W_{J\cap K},\qquad
    d,d'\in W_{K\cap L}.
\]
Then
\[
    x=a(bb')(dd')a'\in W_{J\cap L} W_{J\cap K} W_{K\cap L} W_{J\cap L}\subseteq W_JW_L.
\]
Part~(a), applied to the proper subsets $J$ and $L$, now gives $x=1$.
\end{proof}

We have thus shown that $\Br(W,S)$ satisfies the three-factor avoidance condition \eqref{eq:trivial_intersection_N}. We can then use subgroups of $\Br(W,S)$ to construct controlled quotients.
\begin{thm}
\label{thm:Brunnian-hypertope}
Let $(W,S)$ be a right-angled Coxeter system of finite rank at least two, and
consider its Coxeter chamber system
\[
    \CC=\CC\bigl(W,(\langle r_i\rangle)_{i\in I}\bigr).
\]
Let $H\leq\Br(W,S)$.  Then $\CC/H$ is a hypertope.  If $|I|\geq 3$, it has the same Coxeter diagram as $\CC$.
\end{thm}

\begin{proof}
The subgroup $\Br(W,S)$ is normal in $W$, and
Proposition~\ref{prop:Brun_intersect}(b) shows that it satisfies the
three-factor avoidance condition \eqref{eq:trivial_intersection_N}, with
$G^{(J)}=W_J$.

Since the Coxeter chamber system $\CC$ is connected, thin, and residually-connected, by 
Theorem~\ref{thm:quotient_RC} we have that $\CC/H$ is residually-connected,
and Corollary~\ref{cor:quot_transfer_thin_thick_linearity} implies that it is
thin. Hence, it is a hypertope. When $|I|\geq 3$, the diagram follows from  Corollary~\ref{cor:quot_transfer_thin_thick_linearity}(c). 
\end{proof}

\section{Finding an appropriate subgroup}
\label{sec:AppropriateSub}

Let $(W,S)$ be an irreducible right-angled Coxeter group with $ S= \{r_i\mid i\in I\}$ and Coxeter diagram $\Delta:= \D(W,S)$, and let $\Br:=\Br(W,S)$ be its Brunnian subgroup as defined in the previous section. So far, we have proved that, given a subgroup $H$ of $Br$, the quotient of a right-angled Coxeter chamber system $\C/H$ will always be a hypertope.

The main goal of this section is to prove that we can actually control the automorphism group of the quotient chamber system in order to produce a hypertope of type $2_J$. In particular, our goal will be to prove the following theorem.

\begin{thm}\label{thm:prescribed-normalizer}
Let $J\subsetneq I$.  There exists a nontrivial proper subgroup $H<\Br$ such that $N_W(H)=C_J$.
\end{thm}

We will exploit the fact that $W$ is an acylindrically hyperbolic group as long as $|I| \geq 3$. For definitions and notations about acylindrical actions and acylindrically hyperbolic groups, please refer to Section \ref{sec:AcylindricallyHyperbolic}. For $H < K < G$, we will denote by $\langle\langle H \rangle\rangle_K$ the normal closure of $H$ in $K$. That is, it is the subgroup of $G$ generated by all conjugates of $H$ by elements of $K$.

\begin{lemma}
    If $|I| \geq 3$, then an irreducible right-angled Coxeter group $W$ over $I$ is acylindrically hyperbolic.
\end{lemma}
\begin{proof}
This is a direct consequence of Corollary 2.13 of Minasyan-Osin \cite{minasyan2013acylindrically}. Indeed, the corollary states that an irreducible graph product of non-trivial groups is either virtually cyclic or acylindrically hyperbolic. Any right-angled Coxeter group can be seen as a graph product of groups where each vertex group is cyclic of order two. Beware that the usual convention for graph products is to put an edge to indicate commutation relations. This is the opposite of the convention for Coxeter diagram, where we omit edges to indicate commutation. A graph product is then irreducible when the graph complement is connected. This coincides with our definition of $W$ being irreducible, which is the case by assumption for us. Finally, as soon as $I$ contains at least $3$ elements, the Coxeter group $W$ contains a non-abelian free subgroup, and thus cannot be virtually cyclic. Therefore, $W$ is acylindrically hyperbolic.
\end{proof}

For the remaining of the paper, we will fix a non-elementary acylindrical action of $W$.
We can then utilise all the theory of acylindrically hyperbolic groups to prove Theorem \ref{thm:prescribed-normalizer}. Here is a loose description of the procedure we will follow. Let $N_J = \{H < \Br \mid H \; \text{is normal in } C_J\}$ be the collection of subgroups of $\Br$ that are normal in $C_J$. Then, for any $H \in N_J$, we have $C_J \leq N_W(H)$. But $C_J$ is an index two subgroup of $G$, so $N_W(H)$ must be either $C_J$ or $W$. In order to guarantee that $N_W(H)$ is not the whole group $W$, it is then sufficient to show that $r_i H r_i \neq H$ for a choice of $i \notin J$, so that $r_i \notin C_J$. Suppose now that we have a well-chosen infinite order element $a \in \Br$. Consider $b = a^{r_i} := r_i a r_i$. Form the group $ H =\langle\langle a,b\rangle \rangle_{C_J}$, the normal closure in $C_J$ of the group generated by $a$ and $b$. Then , $H$ belongs to $N_J$. Clearly, we have that $r_i$ normalizes $H$ since $r_i$ conjugates $a$ to $b$ and $b$ to $a$. However it would be reasonable to expect that $a$ can be chosen so that $r_i$ will not normalize the group $\langle \langle a^{m_1}, b^{m_2}\rangle \rangle$ for high enough exponents $m_1$ and $m_2$. This crucial last step will be formalized by using the Dehn filling results of Dahmani-Guirardel-Osin~\cite[Theorem 7.19]{DGO}, presented in Theorem~\ref{thm:DGO_thms}.

We start by checking that the subgroup $\Br$ is sufficiently large.

\begin{lemma}\label{lem:Binfinite}
    If $\Delta$ is connected and $|I| \geq 3$, then the group $\Br$ is an infinite normal subgroup of $W$.
\end{lemma}

\begin{proof}
    The group $\Br$ is normal as it is an intersection of kernels. Suppose for now that $\Br$ is not trivial. Proposition 4.3 of \cite{ParisCoxeter} states that the normalizer $N_W(H)$ of any non-trivial finite subgroup of $W$ must have infinite index in $W$. This directly implies that $W$ cannot have finite non-trivial \emph{normal} subgroups, since their normalizer would be the whole group $W$. Hence, if $\Br$ is non-trivial, it must be infinite.
    
    It remains to show that $\Br$ is not the trivial group. We do so by explicitly constructing an element in $\Br$. Let $ p =p_1-p_2-\cdots-p_{k-1}-p_k$ be a path in the Coxeter diagram $\Delta$ that goes through every vertex, possibly using backtracking. To each $p_i$, let $r_{p_i}$ be the corresponding involution. We construct the nested commutator
    \[
    x := [r_{p_1},[r_{p_2},[\ldots [r_{p_{k-1}},r_{p_k}]]]\ldots]
    \]
    We first show that $x \in \Br$. First of all, every $r_i$ appears at least once since the path $p$ goes through every vertices of $\Delta$. Therefore, deleting any $r_i$ from $x$ creates at least one commutator $[1,y] = 1 = [y,1]$, which then propagates. Hence $x$ becomes trivial after any deletion of a generator, however $x$ is itself non-trivial. Indeed, since we are taking paths along edges labelled $\infty$, no two consecutive generators $r_{p_i}$ and $r_{p_{i+1}}$ will commute. Hence, the expression for $x$ is in the standard normal form for a non-trivial word of graph products.
    
\end{proof}

We now need to find well-behaved loxodromic elements inside of $\Br$. To do so we will use Corollary 5.7 of~\cite{hull2017small}. Note that this corollary considers only actions of $G$ on Cayley graphs $\Gamma(G,\mathcal{A})$. This is not a problem, since by Theorem 1.2 of \cite{OsinAcylindrical}, we know that such an action must exist for any acylindrically hyperbolic group $G$. In the proof of the following lemma, we do not use any properties specific to the choice of action. 

\begin{lemma}\label{lem:existanceLoxodromic}
    There exists a loxodromic $WPD$ element $a \in \Br$ such that $E_W(a) = \langle a \rangle$. 
\end{lemma}
\begin{proof}
    Corollary 5.7 of \cite{hull2017small} states that we can find the loxodromic elements with $E_W(a)=\langle a\rangle$ inside any suitable subgroup of $W$. Since we are working with a fixed acylindrical action of $W$, this loxodromic element will automatically be WPD with respect to that action. We thus only need to prove that $\Br$ is a suitable subgroup of $W$. To do so, we first show that the induced action of $\Br$ in non-elementary.

    A subgroup $H < W$ is called $s$-normal if $|H \cap H^g| = \infty$ for all $g \in W$. The subgroup $\Br$ being normal and infinite, it is $s$-normal. By Lemma~\ref{lem:snormal_non_elem}, the induced action of $\Br$ is non-elementary, so that $\Br$ is also an acylindrically hyperbolic group. 
    
    We will now prove that $\Br$ is suitable. 
    Let $K_W(\Br)$ be the maximal finite subgroup of $W$ normalized by $\Br$.
    A non-elementary subgroup $H\leq G$ is said to be \emph{suitable} if the maximal finite subgroup of $G$ normalized by $H$ is trivial, that is $K_W(\Br)=\{ 1\}$. As $\Br$ is normal, any conjugate of $K_W(\Br)$ is also a finite subgroup normalized by $\Br$. By maximality, all the conjugates must be inside of $K_W(\Br)$, so $K_W(\Br)$ is also normal. As stated in the proof of Lemma~\ref{lem:Binfinite}, $W$ does not have non-trivial finite normal subgroups. Thus $K_W(\Br)=\{ 1\}$ and $\Br$ is suitable.
    We are therefore in the conditions to use Corollary 5.7 of~\cite{hull2017small} and conclude the proof. 
    \end{proof}

We are now ready to prove the main theorem of this section

\begin{proof}[Proof of Theorem \ref{thm:prescribed-normalizer}]
First, we choose some $i \notin J$ and let $r := r_i$ be the associated involutive generator of $W$. Note that $r \notin C_J$. 

By Lemma \ref{lem:existanceLoxodromic}, we can find $a\in \Br$ such that $a$ is a loxodromic $WPD$ element with $E_W(a) = \{g \in W : \exists k,m \neq 0 , g^{-1} a^k g = a^m \} = \langle a \rangle $. Let us define $b := a^r \in \Br$. Since $b$ is a conjugate of $a$, it is also a loxodromic $WPD$ element of $\Br$ with $E_W(b) = \langle b \rangle$. The goal now is to apply Theorem~\ref{thm:DGO_thms} to $\{\langle a\rangle, \langle b\rangle\}$. For that, we need to check that $\{\langle a\rangle, \langle b\rangle\}$ hyperbolically embeds into $C_J$.

First, notice that $a$ and $b$ are non-commensurable in $C_J$. Indeed, for a contradiction, suppose there exists $k \in C_J$ such that 
$$
k^{-1}a^m k = b^n =ra^nr.
$$
Then, setting $g = kr$, we get $g^{-1} a^m g  = a^n$. This implies that $g \in E_W(a) = \langle a \rangle$ and thus $g \in C_J$. However, as $g = kr$ and $k \in C_J$, this would imply $r \in C_J$, a contradiction. 
By the first part of Theorem~\ref{thm:DGO_thms}, the pair $\{\langle a \rangle, \langle b \rangle \}$ hyperbolically embeds into $W$. As $C_J$ is a finite index subgroup of $W$, the pair also hyperbolically embeds in $C_J$.

We can therefore apply the second part of Theorem~\ref{thm:DGO_thms}. Take $N_a = \langle a^{2k} \rangle$ and $N_b = \langle b^{3k} \rangle$ for some sufficiently big integer $k$ that guarantees to avoid the forbidden finite subsets of the hypothesis of the theorem.
Set $H = \langle\langle N_a \sqcup N_b\rangle \rangle_{C_J}$ which is clearly a subgroup of $\Br$. The theorem says that the maps $\langle a \rangle/ \langle a^{2k} \rangle \hookrightarrow C_J / H$ and $\langle b \rangle/ \langle b^{3k} \rangle \hookrightarrow C_J / H$ are both injective. Therefore, in $C_J / H$, the image of $a$ has order exactly $2k$ but the image of $b$ has order $3k$. Therefore, $r$ cannot normalize $H = \langle\langle a^{2k},b^{3k}\rangle\rangle_{C_J}$. Indeed, if that were the case, conjugation by $r$ would go down to an automorphism of $C_J / H$, however it would send an element of order $2k$ to an element of order $3k$, which is clearly not possible.
\end{proof}

\section*{Acknowledgments}
The first author acknowledges financial support from the Actions de Recherche Concertées of the Fédération Wallonie-Bruxelles. The second author is a Postdoctoral Researcher of the Fonds de la Recherche Scientifique -- FNRS.

\section*{Disclosure of use of AI tools}

The proof strategy of Theorem \ref{thm:prescribed-normalizer} is the result of a query to GPT-5.6 Sol. The model's answers allowed the authors to discover and understand the needed concepts from the theory of acylindrically hyperbolic groups, and included a detailed list of citations and theorems. The authors then carefully checked each of the cited works, and reformulated the whole proof according to their own understanding. Large language models (LLMs) were also used sparingly to improve the clarity and readability of certain parts of the manuscript.

\bibliographystyle{ieeetr}
\bibliography{refs}

\end{document}